\documentclass[11pt,reqno]{amsart}
\usepackage{dsfont}

\theoremstyle{definition}
\newtheorem{theorem}{Theorem}
\newtheorem{proposition}[theorem]{Proposition}

\newtheorem{corollary}[theorem]{Corollary}

\newtheorem{lemma}[theorem]{Lemma}

\newtheorem{definition}[theorem]{Definition}
\newtheorem{example}[theorem]{Example}
\newtheorem{remark}[theorem]{Remark}

\usepackage{graphicx}
\usepackage{amsmath}
\usepackage{amssymb}
\usepackage{amsthm}
\usepackage{epsfig}
\usepackage{url}
\usepackage{array}
\usepackage{varwidth}
\usepackage{tikz}
\usetikzlibrary{patterns}
\usepackage{multirow}
\usepackage{textcomp}
\usepackage{bbm}
\usepackage{float}
\usepackage{color}

\newcommand{\mo}[1]{{\color{blue}#1}}
\newcommand{\goaway}[1]{}

\DeclareMathOperator{\tr}{tr}
\renewcommand{\vec}[1]{\mathbf{#1}}

\DeclareMathOperator{\per}{per}
\newcommand{\Sym}{\mathfrak{S}}
\DeclareMathOperator{\sgn}{sgn}
\newcommand{\U}{\mathcal{U}}
\newcommand{\comp}[1]{\overline{#1}}

\newcommand{\Aop}{A^{\text{op}}}
\newcommand{\Dop}{D^{\text{op}}}
\newcommand{\cyc}{\mathrm{cyc}}
\newcommand{\Des}{\mathrm{Des}}

\begin{document}

\title{Revisiting The R\'{e}dei-Berge Symmetric Functions via Matrix Algebra}

\author{John Irving}
\address{Department of Mathematics. St. Mary's University. Halifax, Nova Scotia, Canada. B3H 3C3}
\email{john.irving@smu.ca}

\author{Mohamed Omar}
\address{Department of Mathematics \& Statistics. York University. Toronto, Ontario, Canada. M3J 1P3}
\email{omarmo@yorku.ca}


\begin{abstract}
We revisit the R\'{e}dei-Berge symmetric function $\U_D$ for digraphs $D$, a specialization of Chow's path-cycle symmetric function. Through the lens of matrix algebra, we consolidate and expand on the work of Chow, Grinberg and Stanley, and Lass concerning the resolution of $\U_D$ in the power sum and Schur bases. Along the way we also revisit various results on Hamiltonian paths in digraphs.
\end{abstract}


\thanks{}
\date{\today}
\maketitle

\section{Introduction}
Let $D$ be a digraph on vertex set $[n]:=\{1,2,\ldots,n\}$ and let $\pi = \pi_1 \cdots \pi_n \in \Sym_n$, where here $\pi_i$ is the image $\pi(i)$ of $i$ under $\pi$.  An index $i \in [n-1]$ is said to be a \emph{$D$-descent} of $\pi$ if  $(\pi_i, \pi_{i+1})$ is a directed edge of $D$.  Let $\Des_D(\pi)$ denote the set of all $D$-descents of $\pi$. The \emph{R\'{e}dei-Berge symmetric function} of $D$ is defined by
\begin{equation}
\label{eq:U_D}
	\U_D := \sum_{\pi \in \Sym_n} F_{\Des_D(\pi)},
\end{equation}
where the expansion is in terms of the fundamental quasisymmetric functions
$$
	F_I := \sum_{\substack{1 \leq i_1 \leq \cdots \leq i_n \\  \text{$i_j < i_{j+1}$ for $j \in I$}}} z_{i_1} z_{i_2} \cdots z_{i_n}, \qquad I \subseteq [n-1].
$$

The series $\U_D$ first appears in the work of Chow~\cite{ChowPathCycle}, where it manifests as a specialization of his \emph{path-cycle symmetric function}; this connection will be further detailed below.

More recently, $\U_D$ was studied (and named) by Grinberg and Stanley~\cite{GrinbergStanley}, who demonstrated its connection to well-known results of R\'edei and Berge concerning the number of Hamiltonian paths in digraphs.  For instance, they rederive the classic theorem of R\'{e}dei \cite{Redei} that every tournament $D$ has an odd number of directed Hamiltonian paths, and extend this to new relations on the number of such paths modulo $4$. 

In this article, we re-prove results of Chow \cite{ChowPathCycle}, and Grinberg and Stanley \cite{GrinbergStanley}, on expansions of $\U_D$ in terms of classic symmetric functions. Along the way we rederive formulas for Hamiltonian paths in digraphs that coalesce results of Goulden and Jackson \cite{goulden1981enumeration}, Grinberg and Stanley \cite{GrinbergStanley}, Wiseman \cite{Wiseman}, and others. This is done by unifying the theory from these articles into an overarching matrix algebra framework. The majority of these results hinge on Theorem~\ref{thm:main}, which explicitly writes $\U_D$ in terms of symmetric functions that depend on the entries of the adjacency matrix of $D$ and its complement. Theorem~\ref{thm:main} is ultimately a re-framing of work due to Lass~\cite{Lass} using the algebra of set functions (\emph{fonctions d'ensembles}). Indeed, our use of the linear coefficient operator $\mathfrak{L}_n$, which is critical to our main theory (see Section~\ref{sec:main}), is equivalent to Lass' formalism.

The article is structured as follows: We begin in Section~\ref{sec:prelim} by viewing walks and paths in digraphs through the lens of matrix algebra. Therein we provide proofs of classic results on Hamiltonian paths and cycles in digraphs via their directed adjacency matrices. In turn, these results support Section~\ref{sec:main} which is dedicated to symmetric function expansions of $\U_D$. We recover the power sum expansions of Grinberg and Stanley \cite{GrinbergStanley}, determine general expansions for $\U_D$ in the Schur function basis, and recover positivity results in both these bases. We then show how to lift this theory to the full Chow path-cycle symmetric function.

\section{The Algebra of Walks and Paths}\label{sec:prelim}

\subsection{Preliminaries}
We begin with algebraic preliminaries. For $f \in \mathbb{Q}[[x_1,\ldots,x_n]]$ and $I=\{i_1,i_2,\ldots,i_k\} \subseteq [n]$, we use the notation $\mathfrak{L}_{I}\,f$ to denote the coefficient of $x_{i_1}x_{i_2}\cdots x_{i_k}$ in $f$. In the particular case that $I=[n]$, we abbreviate $\mathfrak{L}_I$ by $\mathfrak{L}_n$. For $I \subseteq [n]$, we denote the complement $[n] \backslash I$ by $I^c$, with $n$ understood from context.  The group of permutations on $I$ is denoted $\Sym_I$, or simply $\Sym_n$ when $I=[n]$.

Let $A$ be an $n \times n$ matrix whose $(i,j)$-entry is denoted $A_{i,j}$. Recall that the permanent $\per(A)$ and  determinant $\det(A)$ of $A$ are defined by
\[
\per(A) = \sum_{\sigma \in \mathfrak{S}_n} \prod_{i=1}^n A_{i,\sigma(i)}, \qquad \det(A)=\sum_{\sigma \in \mathfrak{S}_n} \sgn(\sigma) \prod_{i=1}^n A_{i,\sigma(i)},
\]
where $\sgn(\sigma)$ is the sign of the permutation $\sigma$.  
We denote by $A[S]$  the principal submatrix of $A$ whose rows and columns are indexed by $S \subseteq [n]$, and  let $X=\text{diag}(x_1,\ldots,x_n)$.
The following well-known relations between the permanent and determinant are used   throughout:
\begin{equation}\label{eqn:MacMahon}
\mathfrak{L}_S \det(I+XA) = \det A[S] \qquad\text{and}\qquad \mathfrak{L}_S \det(I-XA)^{-1} = \per A[S],    
\end{equation}
Note that the latter of these is a special case of MacMahon's Master Theorem (see~\cite[Section 1.2.11]{gouldenjacksonbook} and also Lemma~\ref{lem:coefficientextraction}, below).
 
Now recall that if $A$ is $n \times m$ and $B$ is $m \times n$, then Sylvester's determinant identity states that \begin{equation}\label{eqn:Sylvester}
\det(I_{n}+AB)=\det(I_{m}+BA).
\end{equation}
We will make particular use of the special case $m=1$, which gives  
\begin{equation}\label{eqn:Sylvester2}
\det(I_{n}+\vec{u}\vec{v}^T)=1+\vec{v}^T\vec{u}
\end{equation}
for column vectors $\vec{u},\vec{v}$.  We will also make central use of Jacobi's Identity, which states that
\begin{equation}\label{eqn:Jacobi}
\det \exp(A)=\exp \tr(A)    
\end{equation} 
where $\exp(A)=\sum_{k \geq 0} \frac{A^k}{k!}$. These formulae hold for matrices over the appropriate commutative rings; see, e.g., \cite[Section 1.1.10]{gouldenjacksonbook}.

\subsection{Walks and Paths in Digraphs}\label{sec:walks}

A \emph{digraph} $D$ consists of a finite set of vertices $V$ together with a set $E \subseteq V \times V$ of directed edges $(u,v)$. Note this definition forbids parallel edges. The \emph{complement} of $D$, denoted $\overline{D}$, is the digraph on $V$ that has a directed edge $(u,v)$ precisely when $(u,v)$ is not a directed edge in $D$. The \emph{opposite} of $D$, denoted $\Dop$, is the digraph that has a directed edge $(u,v)$ precisely when $(v,u)$ is a directed edge in $D$. If $V' \subseteq V$, then the subgraph of $D$ \emph{induced by $V'$}, denoted $D[V']$, has vertices $V'$ and edges $E' = E \cap (V' \times V')$.

We will work exclusively with digraphs on vertex set $[n]$, for some $n$, or induced subgraphs thereof.  The reader should assume from hereon that every digraph under consideration is of this type.  In particular, the parameter $n$ will universally denote the number of vertices in an ambient digraph.

The \emph{adjacency matrix} of a digraph $D$ on $[n]$  is the  matrix $A(D) \in \{0,1\}^{n \times n}$ whose $(i,j)$ entry is $1$ if and only if $(v_i,v_j)$ is a directed edge in $D$. For brevity we will write $A,\comp{A}$ and $\Aop$ for $A(D), A(\comp{D})$ and $A(\Dop)$, respectively, whenever $D$ is implicit.  Evidently we have $\comp{A} = \vec{1}\vec{1}^T-A$, where $\vec{1}=[1,1,\ldots,1]^{T}$ is the $n$-dimensional all-ones column vector. Moreover it is immediate that $\Aop=A^T$.

A \emph{walk} of \emph{length $k \geq 0$} in  $D$ is a sequence $w=({i_0},\ldots,{i_k})$ of vertices such that $({i_j},{i_{j+1}})$ is a directed edge in $D$ for all $0 \leq j < k$.   We say $w$ is a \emph{path} in  $D$ if all its vertices ${i_0},{i_1},\ldots,{i_k}$ are distinct.  A \emph{cycle} in $D$ is an equivalence class of walks of the form $({i_0},{i_1},\ldots,{i_k},{i_0})$, where ${i_0},\ldots,{i_k}$ are distinct, and two such walks are equivalent if one is a cyclic shift of the other.  A \emph{cycle cover} of $D$ is a set of vertex-disjoint cycles where each vertex appears in exactly one cycle.

We often will need to refer to cycles in  digraphs that are compatible with cycles in permutations. Let $D$ be a digraph on $V \subseteq [n]$ and let $\pi \in \Sym_V$ be a permutation on $V$. Then a cycle $(i_0 \ i_1 \cdots i_k)$ in the disjoint cycle decomposition of $\sigma$ is said to be a \emph{$D$-cycle} if $(i_0,i_1,\ldots,i_k,i_0)$ is a cycle in $D$. 

A convenient way to keep track of walks in  $D$ is through its \emph{walk generating function}.  Let $\gamma_0(D)=1$ and, for any $k \geq 0$, let $\gamma_{k+1}(D)$ be the polynomial in indeterminates $x_1,x_2,\ldots,x_n$ given by
\[
\gamma_{k+1}(D):=\sum_{\substack{(i_0,i_1, \ldots, i_{k})\\ \text{a walk in }D}} x_{i_0}x_{i_1} \cdots x_{i_{k}}.
\]
That is, $\gamma_{k+1}(D)$ accounts for walks of length $k$, which are those containing $k+1$ vertices (counted with multiplicity).  
The walk generating function of $D$ is then the series $W_D(z)$ in $\mathbb{Q}[x_1,x_2,\ldots,x_n][[z]]$ defined by
\[
W_D(z)=\sum_{k \geq 0} \gamma_k(D) z^{k}.
\]
The role of the indeterminates $x_i$ in $W_D(z)$ is left implicit for brevity.  Evidently,  $W_D(z)$ involves only those  $x_i$ for which $i$ is a vertex of $D$. In particular, if $D$ is a digraph on $[n]$ and $V \subseteq [n]$, then $W_{D[V]}(z)$ is therefore obtained from $W_D(z)$ by setting $x_j=0$ for all $j \notin V$.

If $D$ is a digraph on $[n]$ then we can express $W_D(z)$ in a compact form by using its adjacency matrix $A$ and the diagonal matrix  $X=\text{diag}(x_1,x_2,\ldots,x_n)$. An easy induction reveals that the $(i,j)$ entry of $(XA)^{k} X$ is the generating series for walks of length $k$ in $D$ from vertex $i$ to vertex $j$. Thus $\gamma_{k+1}(D)$ is the sum of all entries of this matrix, i.e. $\gamma_{k+1}(D)=\vec{1}^T(XA)^{k}X\vec{1}$.  Letting $\vec{x}=[x_1,\ldots,x_n]^T$, we consequently have
\begin{equation}\label{eqn:wz}
W_D(z)
= 1+\vec{1}^T\left( z\sum_{k \geq 0} (XA)^{k}z^k\right)\vec{x} = 1+z\vec{1}^T(I-zXA)^{-1}\vec{x}.
\end{equation}
The following reformulation of this identity provides a compact expression for $W_D(z)$ that is fundamental to our developments. An equivalent observation was made in \cite{stanleycountingpaths}.

\begin{lemma}\label{lem:walks}
Let $D$ be a digraph on $[n]$ with adjacency matrix $A$. Then 
 $$W_{{D}}(z)=\dfrac{\det(I+zX\comp{A})}{\det(I-zX{A})}.$$
Consequently we have $W_{\comp{D}}(z)=(W_D(-z))^{-1}$, and when $D$ is acyclic $W_D(z)=\det(I+zX\comp{A})$.
\end{lemma}

\begin{proof}
We have
\begin{align*}
\det(I+zX\comp{A})
&= \det(I+zX(\vec{1}\vec{1}^T-{A}))\\
&=\det(I-zX{A}+z\vec{x}\vec{1}^T)\\
&=\det(I-zXA)\det(I+z(I-zXA)^{-1}\vec{x}\vec{1}^T) \\
&=\det(I-zXA)(1+z\vec{1}^T(I-zX{A})^{-1} \vec{x}),
\end{align*}
where the final equality comes from Sylvester's identity (\ref{eqn:Sylvester2}) applied to the pair $\vec{u}=z(I-zXA)^{-1}\vec{x}$ and $\vec{v}=\vec{1}$.
From Equation~(\ref{eqn:wz}) we then get 
\[
W_D(z)=1+z\vec{1}^T(I-zX{A})^{-1} \vec{x}=\frac{\det(I+zX\comp{A})}{\det(I-zXA)}.
\]  
It immediately follows that $W_{\comp{D}}(z)=(W_D(-z))^{-1}$. Finally, if $D$ is acyclic then we can relabel its vertices so as to make $A$ strictly upper triangular (i.e. topologically ordered), which gives $\det(I-zXA)=1$ and thus $W_D(z)=\det(I+zX\comp{A})$. 
\end{proof}
\newcommand{\ham}{\mathrm{ham}}

Recall a path of length $n-1$ in a digraph $D$ on $[n]$ is said to be \emph{Hamiltonian}, meaning it encounters every vertex in $D$.  For convenience we also consider the empty digraph on zero vertices to have a single Hamiltonian path.   The number of Hamiltonian paths in $D$ is denoted  $\mathrm{ham}(D)$.  We now observe that Lemma~\ref{lem:walks} yields a permanent-determinant formula for this statistic. 

\begin{proposition}
\label{prop:hamps}
    The number of Hamiltonian paths in a digraph $D$ on $[n]$ is given by
    $$
    \ham(D)=
    \sum_{S \subseteq [n]}
\det \comp{A}[S] \cdot \per {A}[S^c].
    $$
\end{proposition}

\begin{proof}
From the definition of $W_D(z)$ we have $\ham(D)=\mathfrak{L}_n\,W_{{D}}(1)$. By  Lemma~\ref{lem:walks} this is
\begin{align*}
\mathfrak{L}_n\, \det(I+X\comp{A})\det(I-X{A})^{-1} 
&=
\sum_{S \subseteq [n]}
\mathfrak{L}_{S} \det(I+X\comp{A}) \cdot \mathfrak{L}_{S^c} \det(I-X{A})^{-1},
\end{align*}
and the result follows from Equation~\eqref{eqn:MacMahon}.
\end{proof}

It is clear from  definition that for $\sigma \in \Sym_n$ the product $\prod_{i=1}^n A_{i,\sigma(i)}$ equals $1$ if every cycle in the disjoint cycle decomposition of $\sigma$ is a $D$-cycle, and equals $0$ otherwise. Consequently, $\per(A)$ is the number of cycle covers of $D$. We can therefore use Proposition~\ref{prop:hamps} to recover the following result of Wiseman \cite{Wiseman}.
\begin{corollary}\label{cor:wiseman}\cite[Theorem 2.1]{Wiseman}
Suppose $D$ is an acyclic digraph. Then the number of Hamiltonian paths in $\comp{D}$ is equal to the number of cycle covers of $\comp{D}$.
\end{corollary}
\begin{proof}
Since $D$ is acyclic, $\det A[S] =0$ for $S \neq \emptyset$, so by Proposition~\ref{prop:hamps}, $\ham(\comp{D})=\per(\comp{A})$. The result follows. 
\end{proof}

The following result of Berge is immediate from Proposition~\ref{prop:hamps} upon observing that $\det(M) \equiv \per(M) \pmod{2}$ for any integer matrix $M$.

\begin{corollary} \cite[Section 10.1]{Berge91}
    For any digraph $D$ we have $\ham(D) \equiv \ham(\comp{D}) \pmod{2}$.
\end{corollary}

Not all digraphs have Hamiltonian paths,  but certain  classes are known to have at least one. For instance, a classic result of R\'{e}dei \cite{Redei} establishes that any tournament has an odd number of Hamiltonian paths. Recall that a digraph $D$ is a \emph{tournament} if for every pair of distinct vertices $i,j$, exactly one of $(i,j)$ and $(j,i)$ is a directed edge (and $(i,i)$ is not an edge for any $i$). Here we offer a different proof of R\'{e}dei's result using Lemma~\ref{lem:walks}.

\begin{theorem}\label{thm:Redei}\cite{Redei}
If $D$ is a tournament on $[n]$ then $\mathrm{ham}(D)$ is odd.    
\end{theorem}

\begin{proof}
As before we write $\ham(D) = \mathfrak{L}_n\,W_D(1)$.  Since $D$ is a tournament we have $\comp{A}=A^T+I$ and thus
$$
	\det(I+X\comp{A})=\det(I+X(A^T+I))=\det(I+(A^T+I)X)=\det(I+XA+X),
$$
the middle equality by Sylvester's Identity \eqref{eqn:Sylvester}. We are concerned only with the parity of $\mathfrak{L}_n\,W_D(1)$, so we now work over the quotient ring $\mathbb{F}_2[x_1,\ldots,x_n] /\langle{x_1^2,\cdots,x_n^2\rangle}$ in which $X^2=0$ and $XA =-XA$. From this,  Lemma~\ref{lem:walks} gives
\begin{align*}
W_D(1) = \det(I+XA+X)\det(I-XA)^{-1} 
&= \det(I-XA+X)\det(I-XA)^{-1} \\
&= \det(I+X(I-XA)^{-1}) \\
&= \det(I+X(I + XA + (XA)^2 + \cdots)) \\
&= \det(I+X),
\end{align*}
the last equality since $X^2=0$. This yields $\mathfrak{L}_n\,W_D(1) = 1$, completing the proof.
\end{proof}

Proposition~\ref{prop:hamps} is similar to known expressions due to Goulden and Jackson~\cite{goulden1981enumeration} and Liu~\cite{Liu} for the number of Hamiltonian cycles in a digraph $D$. (Recall that a cycle in $D$ is said to be \emph{Hamiltonian} if it encounters every vertex; the empty digraph having no Hamiltonian cycles, and a digraph on one vertex having a Hamiltonian cycle if and only if it contains a loop.)  Although peripheral to our study, we pause here to show how these results can be derived from similar matrix analyses.  
Part (a) of following proposition appears as \cite[Theorem 4.1]{goulden1981enumeration}, and is proven there using Lagrangian methods. It also appears in \cite{Gnang}.  Part (b) is trivially equivalent to Theorem 2 of~\cite{Liu}, where it is derived via the Matrix-Tree Theorem.

\begin{proposition}
\label{prop:GJ}
Let $D$ be a digraph on $[n]$.  For any $i \in [n]$, the number of Hamiltonian cycles in  $D$ is given by the following equivalent expressions: 
\begin{enumerate}
\item[(a)]
$\displaystyle
\sum_{S \subseteq [n]\setminus\{i\}}
 (-1)^{|S|} \det A[S] \cdot \per A[S^c].
$
\item[(b)]
$\displaystyle
\frac{1}{n}\sum_{S \subseteq [n]}
 (-1)^{|S|} |S^c|\,\det A[S] \cdot \per A[S^c].
$
\end{enumerate}
\end{proposition}

\newcommand{\numham}{\mathrm{ham}^{\circ}(D)}
\begin{proof}
Let $\numham$ denote the number of Hamiltonian cycles in $D$. The $(i,i)$-entry of $(XA)^k$ is the generating function for  walks of length $k$ in $D$ that start and end with $i$.  
So let us set 
$$
 H := I + XA + (XA)^2 + (XA)^3 + \cdots = (I-XA)^{-1}
 = \frac{\mathrm{adj}(I-XA)}{\det(I-XA)},
$$
where $\mathrm{adj}(\cdot)$ denotes the adjugate, i.e. classical adjoint. Then for any $i \in [n]$ we have  
\begin{align*}
\numham =\mathfrak{L}_n\, H_{ii}
&=
\mathfrak{L}_n\, \mathrm{cof}_{ii}(I-XA) \det(I-XA)^{-1} \\
&=
\sum_{S \subseteq [n]}
 \mathfrak{L}_{S}\, \mathrm{cof}_{ii}(I-XA) \ \mathfrak{L}_{S^c} \det(I-XA)^{-1},
 \end{align*}
where $\mathrm{cof}_{ii}$ is the $(i,i)$-cofactor. Expression (a) immediately follows using Equation \eqref{eqn:MacMahon}. Averaging (a) over all $i \in [n]$ then yields (b) as follows:
\begin{align*}
\frac{1}{n} \sum_{i=1}^n \sum_{S \subseteq [n]\setminus\{i\}}
 (-1)^{|S|} \det A[S] \cdot \per A[S^c] &= \frac{1}{n} \sum_{S \subseteq [n]} \sum_{i \notin S} 
 (-1)^{|S|} \det A[S] \cdot \per A[S^c] \\
 &= \frac{1}{n} \sum_{S \subseteq [n]} (-1)^{|S|} |S^c| \det A[S] \cdot \per A[S^c].
\end{align*}
\end{proof}

We note that Proposition~\ref{prop:hamps} can be derived from Proposition~\ref{prop:GJ}(a) by applying the latter to the digraph $D'$ that is obtained from $D$ by adding a new vertex $0$ and a directed edge to and from every vertex in $D$. The number of Hamiltonian cycles in $D'$ is then the number of Hamiltonian paths in ${D}$. The reduction from $A(D')$ to $A(D)$ can be accomplished via Schur complementation.

\section{Symmetric Function Expansions}\label{sec:main}

\subsection{Preliminaries}
We use standard notation for partitions of integers, symmetric functions, and representation theory of the symmetric group, as introduced in Macdonald \cite{macdonald1998symmetric}. We remind the reader of a few key constructs and auxiliary definitions.

Throughout, we work in the ring of symmetric functions $\Lambda_{\vec{z}}=\mathbb{Q}[[z_1,z_2,\ldots]]^{\mathfrak{S}}$ in the indeterminates $\vec{z}=(z_1,z_2,\ldots)$. For a partition $\lambda$, we let $r_{\lambda}! := r_1! r_2! \cdots$, where $r_i$ is the number of parts of $\lambda$ equal to $i$.  In addition to the usual symmetric functions $m_{\lambda},e_{\lambda},h_{\lambda},p_{\lambda},s_{\lambda}$, we employ the \emph{augmented monomial symmetric functions} defined by $\tilde{m}_{\lambda}:= r_{\lambda}!  m_{\lambda}$.   For $f$ in one of the families $\{m_{\lambda},e_{\lambda},h_{\lambda},p_{\lambda},s_{\lambda},\tilde{m}_{\lambda}\}$, we denote by $f_{\cyc(\sigma)}$ the symmetric function $f_{\lambda}$ where $\lambda=(\lambda_1,\lambda_2,\ldots,\lambda_{\ell})$ is the cycle type of $\sigma$. For example, if $\sigma = (1 \ 6 \ 5)(2)(3 \ 4)(7 \ 9)(8) \in \mathfrak{S}_9$ then $p_{\cyc(\sigma)} = p_3p_2^2p_1^2$.

Let $H_{\vec{z}}(t),E_{\vec{z}}(t),P_{\vec{z}}(t) \in \Lambda_{\vec{z}}[[t]]$ be the ordinary generating series in $t$ for $\{h_i\}_{i \geq 0}, \{e_i\}_{i \geq 0}$ and $\{p_{i+1}\}_{i \geq 0}$ respectively. One can readily see that 
\begin{equation}\label{eqn:involution}
H_{\vec{z}}(t)=\sum_{j \geq 0} h_jt^j = \prod_{i \geq 1} \frac{1}{1-z_it}, \hspace{0.2in} E_{\vec{z}}(t)=\sum_{j \geq 0} e_jt^j = \prod_{i \geq 1} (1+z_it).
\end{equation}
Furthermore, one can check that $P_{\vec{z}}(t) = \frac{d}{dt} \log H_{\vec{z}}(t)$ and by Equation~(\ref{eqn:involution}), $P_{\vec{z}}(-t) = \frac{d}{dt} \log E_{\vec{z}}(t)$. By integrating then exponentiating we get 
\begin{equation}\label{eqn:powersum}
H_{\vec{z}}(t) = \exp \left( \sum_{i \geq 1} p_i \frac{t^i}{i} \right), \ E_{\vec{z}}(t) = \exp \left( \sum_{i \geq 1} (-1)^{i-1} p_i \frac{t^i}{i} \right).
\end{equation}
The \emph{fundamental involution} $\omega: \Lambda_{\vec{z}} \to \Lambda_{\vec{z}}$ is the algebra homomorphism defined on the generating set $\{e_n\}_{n \geq 1}$ by $\omega(e_n)=h_n$ for each $n$. From Equation~(\ref{eqn:involution}) we see $\sum_{i=0}^n (-1)^i e_ih_{n-i}=0$ for $n>0$ and from this $\omega(h_n)=e_n$ for all $n$.

We see that (\ref{eqn:involution}) and (\ref{eqn:powersum}) can be used to express one class of symmetric functions in terms of another.  We present a few additional such relations pertinent to our discussion. Recall that given a partition $\lambda=(\lambda_1,\lambda_2,\ldots,\lambda_{\ell})$, its \emph{conjugate} is the partition $\lambda^T=(\lambda^T_1,\lambda^T_2,\ldots)$ where $\lambda^T_i$ is the number of indices $j$ for which $\lambda_j \geq i$. For instance, if $\lambda=(4,3,2,2,1)$ then $\lambda^T=(5,4,2,1)$. The classical Jacobi-Trudi identities express the Schur function $s_{\lambda}$  in terms of homogeneous and elementary symmetric functions as follows:
\begin{align}
\label{eqn:jacobitrudi}
s_{\lambda} = \det [h_{\lambda_i-i+j}] = \det [e_{\lambda^T_i-i+j}].
\end{align}
Here, $h_k=0$ (similarly $e_k=0$) if $k < 0$. The Cauchy identity asserts that in $\mathbb{Q}[[\vec{y},\vec{z}]]$ we have 
\begin{align}
\label{eqn:cauchy}
\sum_{\lambda \vdash n} h_{\lambda}(\vec{y})m_{\lambda}(\vec{z})
=
\sum_{\lambda \vdash n} s_{\lambda}(\vec{y}) s_{\lambda}(\vec{z})
\end{align}
Furthermore it is well-known that $h_{\lambda} = \sum_{\mu} K_{\mu,\lambda}s_{\mu}$ where the coefficients $K_{\mu,\lambda}$ are nonnegative integers (the so-called Kostka numbers).  Yet another classic result shows that if $\mu \vdash n$ then 
\begin{equation}\label{eqn:powertoschur}
p_{\mu} = \sum_{\lambda \vdash n} \chi^{\lambda}(\mu) \cdot s_{\lambda},
\end{equation}
where $\chi^{\lambda}(\mu)$ is the evaluation of the irreducible character $\chi^{\lambda}$ of $\mathfrak{S}_n$ indexed by $\lambda$ evaluated at any permutation whose cycle type is given by $\mu$. 

The characters of representations of $\Sym_n$ show up in another useful context for us. Given an $n \times n$ matrix $A$ and a partition $\lambda \vdash n$, the \emph{immanant} of $A$ indexed by $\lambda$ is the multilinear expression $\text{Imm}_{\lambda}(A) = \sum_{\sigma \in \mathfrak{S}_n} \chi^{\lambda}(\sigma) \prod_{i=1}^n A_{i,\sigma(i)}.$ When $\lambda=(n) \vdash n$ then $\chi^{\lambda}$ is the trivial character given by $\chi^{\lambda}(\sigma)=1$ for all $\sigma \in \mathfrak{S}_n$ so $\text{Imm}_{\lambda}(A)=\text{per}(A)$. Similarly when $\lambda=(1,1,\ldots,1)$ then $\chi^{\lambda}$ is the sign character given by $\chi^{\lambda}(\sigma)=\sgn(\sigma)$ for all $\sigma \in \mathfrak{S}_n$ so $\text{Imm}_{\lambda}(A)=\text{det}(A)$.

In our study we will be concerned with when symmetric functions expand with nonnegative coefficients with respect to a given basis. We say $f \in \Lambda_{\vec{z}}$ is \emph{$p$-positive} if it can be written as a nonnegative linear combination of power sum symmetric functions. We say $f$ is \emph{Schur-positive} if it can be written as a nonnegative linear combination of Schur functions.

\renewcommand{\path}[1]{\mathrm{path}(#1)}
\newcommand{\cycle}[1]{\mathrm{cycle}(#1)}
\newcommand{\augm}{\tilde{m}}
\subsection{Path-Cycle and R\'{e}dei-Berge Symmetric Functions}
Let $D$ be a digraph on $I \subseteq [n]$.
A \emph{path-cycle cover} of  $D$ is a spanning subgraph  of $D$ comprised of a vertex-disjoint union of paths and cycles.  Such a subgraph $S$ induces partitions $\path{S}$ and $\cycle{S}$ whose union is a partition of $|I|$ whose parts are the sizes (i.e. the number of vertices) in their respective paths and cycles.  We say $S$ is a \emph{path cover} (respectively, \emph{cycle cover}) of $D$ if it contains only paths (resp. cycles).

In~\cite{ChowPathCycle}, Chow defines the \emph{path-cycle symmetric function} of $D$ by
$$
\Xi_D(\vec{z},\vec{y})
= \sum_S  \tilde{m}_{\path{S}}(\vec{z})\, p_{\cycle{S}}(\vec{y}),
$$
where the sum extends over all path-cycle covers $S$ of $D$. For example, let $D$ be the digraph on vertex set $[3]$ with edge set $\{(1,1),(1,3),(3,2)\}$. Both $D$ and its complement $\comp{D}$ are displayed in Figure~\ref{fig:pathcycle}, and for these  we have 
\begin{align*}
\Xi_{{D}}(\vec{z},\vec{y}) 
&=
\augm_{1^3} + 2\augm_{21} + \augm_{3}+\augm_{1^2}p_1+\augm_2p_1 \\
\Xi_{\comp{D}}(\vec{z},\vec{y}) 
&= \augm_{1^3} + 4\augm_{21} + 3\augm_3
+  \augm_1p_1^2 + 3 \augm_2p_1 + 2  \augm_{1^2}p_1 + p_1p_2+\augm_1p_2+p_3 
\end{align*}
where the $\augm's$ and $p's$ are in $\Lambda_{\vec{z}}$ and $\Lambda_{\vec{y}}$ respectively. For instance, in $\Xi_{\comp{D}}$, the $4\augm_{21}$ accounts for the four path covers $\{12,3\},\{21,3\},\{23,1\},\{31,2\}$ whereas $3\augm_2p_1$ accounts for the path-cycle covers $\{12,33\},\{21,33\},\{31,22\}$. 

It is easy to verify that $\U_D$ is the following evaluation of $\Xi_{\comp{D}}$, as given by Chow.
\begin{proposition}\cite[Proposition 7]{ChowPathCycle} \label{prop:U_D_path}
For a digraph $D$ on $[n]$ we have
$\U_D=\Xi_{\comp{D}}(\vec{z},0).$  That is, 
\[
\U_D = \sum_S \tilde{m}_{\path{S}} = 
\sum_P z_1^{|P_1|} z_2^{|P_2|} \cdots
\]
where the first sum extends over all path covers $S$ of $\comp{D}$, and the second over all sequences $P=(P_1,P_2,\ldots)$ of vertex-disjoint (possibly empty) paths that cover $\comp{D}$.
\end{proposition}
So for instance, for the digraph $D$ in Figure~\ref{fig:pathcycle} and its complement,
\begin{equation}\label{eqn:monomialequation}
\U_D = \augm_{1^3} + 4\augm_{21} + 3\augm_3, \qquad \U_{\comp{D}} = \augm_{1^3} + 2 \augm_{21} + \augm_{3}
\end{equation}
We can express these in terms of power sum symmetric functions as follows:
\begin{equation}\label{eqn:powersumequation}
\U_D = p_{1^3}+p_{21}+p_3,  \qquad \U_{\comp{D}} = p_{1^3}-p_{21}+p_3 
\end{equation}
\begin{figure}[t]
\begin{center}
    \includegraphics[height=2.5cm]{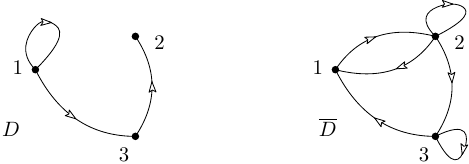}
\end{center}
\caption{A digraph with $D$ and its complement $\comp{D}$.}\label{fig:pathcycle}
\end{figure}

Proposition~\ref{prop:U_D_path} together with Lemma~\ref{lem:walks} will allow us to expand $\U_D$ in terms of power sum symmetric functions in Theorem~\ref{thm:main}. To start, we need the following lemma. 

\newcommand{\SymC}[1]{\mathfrak{C}_{#1}}

\begin{lemma}\label{lem:coefficientextraction}
   For an $n \times n$ matrix $A$ and $I \subseteq [n]$, we have
   \begin{enumerate}
    \item[(a)] $\displaystyle \mathfrak{L}_I\det H_{\vec{z}}(XA) =\sum_{\sigma \in \mathfrak{S}_I} p_{\cyc(\sigma)} \prod_{i \in I} {A}_{i,\sigma(i)}$,
    \item[(b)] $\displaystyle \mathfrak{L}_I\det E_{\vec{z}}(XA)  =\sum_{\sigma \in \mathfrak{S}_I} p_{\cyc(\sigma)} \sgn(\sigma) \prod_{i \in I} A_{i,\sigma(i)}$.
   \end{enumerate}
\end{lemma}
\begin{proof}
 By Jacobi's identity \eqref{eqn:Jacobi} and Equation~(\ref{eqn:powersum}) we have
\begin{align}
\det H_{\vec{z}}(XA) 
= \det \exp \left( \sum_k \frac{1}{k} p_k(XA)^k \right)
=  \exp \left( \sum_k \frac{1}{k} p_k   \tr(XA)^k\right).
\label{eq:H_p_expansion} 
\end{align}
The computation of $\mathfrak{L}_I \det H_{\vec{z}}(XA)$ is  unaffected by setting $x_j^2=0$ for $j \in I$ and $x_{j}=0$ for $j \not \in I$. (More formally, we are passing to the quotient by the ideal generated by $\{x_i^2 \colon i \in I\} \cup \{x_j \colon j \notin I\}$.)

 Under these reductions we have
 \begin{align*}
    \frac{1}{k} \tr(XA)^k 
    &= 
    \frac{1}{k} \sum_{\substack{\text{distinct} \\ j_1,\ldots,j_k \in I}} 
    x_{j_1}A_{j_1,j_2} \cdot x_{j_2} A_{j_2,j_3} \cdots x_{j_k}A_{j_k,j_1} 
    =
    \sum_{\substack{J \subseteq I \\ |J|=k}} 
    \sum_{\substack{ \sigma \in \SymC{J}}} 
    \prod_{j \in J} x_j A_{j,\sigma(j)},
\end{align*}
where $\SymC{J}$ denotes the set of full cycles on $J$ (i.e.  $\sigma \in \Sym_J$ composed of a single cycle). From Equation~\eqref{eq:H_p_expansion}, this gives
\begin{align*}
\det H_{\vec{z}}(XA) &= \exp \left( \sum_{\emptyset \subsetneq J \subseteq I} p_{|J|} \sum_{\sigma \in \SymC{J}} \prod_{j \in J} x_j A_{j,\sigma(j)} \right) \\
&= \prod_{\emptyset \subsetneq J \subseteq I} \exp \left( p_{|J|} \sum_{\sigma \in \SymC{J}} \prod_{j \in J} x_j A_{j,\sigma(j)}  \right). 
\end{align*}
Since $\exp(u) = 1+u$ when $u^2 = 0$, it follows that
\begin{align*}
   \mathfrak{L}_I \det H_{\vec{z}}(XA) 
&= \mathfrak{L}_I \prod_{\emptyset \subsetneq J \subseteq I} \left(1 + p_{|J|} \sum_{\sigma \in \SymC{J}} \prod_{j \in J} x_j A_{j,\sigma(j)} \right) \\
&= \sum_{\substack{\emptyset \subsetneq J_1,J_2,\ldots,J_k \subseteq I \\ J_1 \cup J_2 \cup \cdots \cup J_k = I}} \prod_{s=1}^k \left( p_{|J_s|} \sum_{\sigma \in \SymC{J_s}} \prod_{j \in J_s} A_{j,\sigma(j)}\right).
\end{align*}
Identity (a) now comes by considering the disjoint cycle decomposition of a given $\sigma \in \Sym_I$.  Statement (b) can be derived similarly, with $\sgn(\sigma)$ coming from the contribution $(-1)^{k-1}p_k$ for each $k$-cycle in the disjoint cycle decomposition of $\sigma$.
\end{proof}

Observe that Equation~\eqref{eqn:MacMahon} is recovered from  Lemma~\ref{lem:coefficientextraction} by setting $I=S$ and $\vec{z}=(1,0,0,\ldots)$.  Similarly, setting $z_i=1$ for $1 \leq i \leq \alpha$ and $z_i=0$ for $i > \alpha$ yields an analogue of Equation~\eqref{eqn:MacMahon} for the \emph{$\alpha$-permanent}; see~\cite{foata1988laguerre,vere1988generalization}.

\begin{corollary}\label{cor:cyclecover}
   Let $D$ be a digraph on $[n]$ with adjacency matrix $A$ and let $I \subseteq [n]$. Then
   $$
   \mathfrak{L}_I \det H_{\vec{z}}(XA)
   = \sum_{S} p_{\cycle{S}}, 
   $$
   where the sum extends over all cycle covers $S$ of the subgraph $D[I]$ induced by $I$.
\end{corollary}

\begin{proof}
    Apply Lemma~\ref{lem:coefficientextraction}(a), noting that $\prod_i A_{i,\sigma(i)}$ is 1 if all cycles of $\sigma$ are $D[I]$-cycles and 0 otherwise.
\end{proof}

We now expand $\U_D$ in terms of power sum symmetric functions.   

\begin{theorem}\label{thm:main}
Let $D$ be a digraph on $[n]$. The following are equivalent expressions for $\U_D$:
\begin{enumerate}
\item[(a)] $\displaystyle \U_D = \mathfrak{L}_n \ \det H_{\vec{z}}(X\overline{A}) \cdot \det E_{\vec{z}}(XA)$,
\item[(b)]  $\displaystyle \U_D
 = \mathfrak{L}_n\, \exp  \Big( \sum_k \frac{1}{k}p_k\big(\tr (X\comp{A})^k + (-1)^{k-1} \tr(XA)^k\big)\Big)$,
\item[(c)] $\displaystyle \mathcal{U}_D =
 \sum_{I \subseteq [n]} \left( \sum_{\sigma \in \mathfrak{S}_{I^c}} p_{\cyc(\sigma)} \prod_{i \in I^c} \overline{A}_{i,\sigma(i)} \right) \left( \sum_{\tau \in \mathfrak{S}_{I}} p_{\cyc(\tau)} \sgn(\tau) \prod_{i \in I} A_{i,\tau(i)} \right).$
\end{enumerate}
\end{theorem}

\begin{proof}
Since $W_{\comp{D}}(z)$ is the generating series for walks in $\comp{D}$, Proposition~\ref{prop:U_D_path} implies \[\U_D=\mathfrak{L}_n\, W_{\comp{D}}(z_1)W_{\comp{D}}(z_2)\cdots.\] From this and Lemma~\ref{lem:walks} we get
\begin{align}
    \mathcal{U}_D  
    &= \mathfrak{L}_n\, \prod_{i} 
    \frac{\det( I+z_i XA) }{\det(I-z_iX\comp{A})} \label{eqn:UD_WD} \\
    &=\mathfrak{L}_n \det \prod_i (I-z_iX\comp{A})^{-1} \cdot \det \prod_i \left(I+z_iXA\right) \notag\\
    &= \mathfrak{L}_n\, \det H_{\vec{z}}(X\comp{A}) \det E_{\vec{z}}(XA), \notag
\end{align}
where the final equality is due to Equation~(\ref{eqn:involution}). Now (b) follows from Equation~\eqref{eq:H_p_expansion} in the proof of Lemma~\ref{lem:coefficientextraction} (and its analogue for $E_{\vec{z}}(XA)$) and (c) follows immediately from (a) by the same lemma.
\end{proof}

For example, recall the digraph $D$ from Figure~\ref{fig:pathcycle}. We see that
\[
A = \begin{pmatrix}
1 & 0 & 1 \\
0 & 0 & 0 \\
0 & 1 & 0 
    \end{pmatrix}, \ \  
\comp{A} = \begin{pmatrix}
0 & 1 & 0 \\
1 & 1 & 1 \\
1 & 0 & 1 
\end{pmatrix}.
\]
Thus the only non-zero contributions to the sum in Theorem~\ref{thm:main}(c)  occur when $I=\{1\}$ or $I=\emptyset$. When $I=\{1\}$ the contribution comes from $\tau=(1), \sigma=(2)(3)$, and the contribution itself is $p_{\cyc((1)(2)(3))}=p_{1^3}$. When $I=\emptyset$, there is one contribution from each of $\sigma=(12)(3)$ and $\sigma=(123)$, namely $p_{21}$ and $p_3$, respectively. Altogether we find $\U_D=p_1^3+p_{21}+p_3$, in accord with Equation~(\ref{eqn:powersumequation}).

Chow \cite{ChowPathCycle} and independently Grinberg and Stanley \cite{GrinbergStanley} show that the fundamental involution $\omega$ acts on $\U_D$ by mapping it to $\U_{\comp{D}}$, as can be observed, for example, in Equation~(\ref{eqn:powersumequation}). We now see that this follows directly from  Theorem~\ref{thm:main}. 

\begin{proposition}\label{prop:involution}
For a digraph $D$ on $[n]$, $\U_{\overline{D}}=\omega(\U_D)$.
\end{proposition}

\begin{proof}
This follows directly from Theorem~\ref{thm:main}(a) and the observation that $$\omega(\det H_{\vec{z}}(X\comp{A}) \det E_{\vec{z}}(XA)) = \det \omega H_{\vec{z}}(X\comp{A}) \det \omega E_{\vec{z}}(XA)=\det E_{\vec{z}}(X\comp{A})\det H_{\vec{z}}(XA).$$
\end{proof}

Gruji\'{c} and Stojadinovi\'{c} \cite{redeibergehopfalgebra} give a related result showing that  $\U_{D^{op}}=\U_D$.

\begin{proposition}\cite[Proposition 4.8]{redeibergehopfalgebra}\label{prop:opp}
For a digraph $D$ on $[n]$, $\U_{D^{op}}=\U_D$.    
\end{proposition}
\begin{proof}
The result follows from Theorem~\ref{thm:main}(c)  by observing that a cycle $(i_0,i_1,\ldots,i_k,i_0)$ is in $D$ (respectively, $\comp{D}$) if and only if the cycle $(i_0,i_k,i_{k-1},\ldots,i_1,i_0)$ is in $D^{op}$ (resp. $\comp{D}^{op}$). This can also be seen directly from Proposition~\ref{prop:U_D_path}, since path covers of $D^{op}$ are evidently obtained by reversing the paths in covers of $D$.
\end{proof}

\subsection{Power Sum Expansions}
We now apply Theorem~\ref{thm:main} to recover several theorems from the literature on the expansion of $\U_D$ in the power sum basis. We particularly investigate digraphs $D$ for which $\U_D$ is $p$-positive, and interpret their coefficients combinatorially where possible. For convenience we borrow the following nomenclature  from~\cite{GrinbergStanley}.

\begin{definition}\cite{GrinbergStanley}
 We let $\mathfrak{S}_I(D)$ (respectively  $\mathfrak{S}_I(D,\overline{D})$) denote the set of permutations in $\mathfrak{S}_I$ whose nontrivial cycles are all $D$-cycles (respectively all $D$-cycles or $\overline{D}$-cycles). We simply write $\Sym(D)$ and $\Sym(D,\comp{D})$ for $\Sym_{[n]}(D)$ and $\Sym_{[n]}(D,\comp{D})$ respectively. 
\end{definition}

The following theorem of Grinberg and Stanley gives an explicit expansion of $\U_D$ in terms of power sum symmetric functions.

\begin{corollary}\label{cor:U_Dpowersum}\cite[Theorem 1.3]{GrinbergStanley}
Let $D$ be a digraph on $[n]$. For $\sigma \in \mathfrak{S}_n$, define
\[
\phi(\sigma) := \sum_{\substack{\gamma \in \mathrm{Cycs}(\sigma) \\ \text{$\gamma$ a $D$-cycle}}} (\ell(\gamma)-1),
\]
where $\mathrm{Cycs}( \sigma)$ denotes  the set of cycles in the disjoint cycle decomposition of $\sigma$ and  $\ell(\gamma)$ is the length of $\gamma \in \mathrm{Cycs}(\sigma)$. Then
\[
\U_D = \sum_{\sigma \in \mathfrak{S}(D,\overline{D})}(-1)^{\phi(\sigma)} p_{\mathrm{cyc}(\sigma)}.
\]
\end{corollary}

\begin{proof}
By  Theorem~\ref{thm:main}(c), 
\[
\U_D = \sum_{I \subseteq [n]} \sum_{(\sigma,\tau) \in \mathfrak{S}_{I^c} \times \mathfrak{S}_I} p_{\mathrm{cyc}(\sigma)} p_{\mathrm{cyc}(\tau)} \sgn(\tau) \prod_{i \in I^c} \overline{A}_{i,\sigma(i)} \prod_{i \in I} A_{i,\tau(i)}.
\]
The non-zero summands in this expression come from pairs $(\sigma,\tau) \in\mathfrak{S}_{I^c} \times \mathfrak{S}_I$ for which every cycle in $\sigma$ is a $\comp{D}$-cycle and every cycle in $\tau$ is a $D$-cycle. The contribution to the sum from such a pair is $p_{\mathrm{cyc}(\sigma)} p_{\mathrm{cyc}(\tau)} \sgn(\tau) = p_{\mathrm{cyc}(\rho)}\sgn(\tau)$ where $\rho$ is the permutation in $\mathfrak{S}_n$ obtained by concatenating $\sigma$ and $\tau$. Finally, $\sgn(\tau)$ is precisely $(-1)^{\phi(\rho)}$ and the result follows.
\end{proof}

As an example, consider the digraph $D$ in Figure~\ref{fig:pathcycle}. There are exactly three permutations in $\mathfrak{S}(D,\overline{D})$, namely
\[
\sigma=(1)(2)(3) \hspace{0.2in} \sigma'=(12)(3) \hspace{0.2in} \sigma''=(123).
\]
Note that $\sigma$ has one $D$-cycle (of length $1$), whereas $\sigma'$ and $\sigma''$ have no $D$-cycles. Therefore, $\phi(\sigma)=\phi(\sigma')=\phi(\sigma'')=0$. We deduce then that $\U_D=p_{1^3}+p_{21}+p_3$, in agreement with Equation~(\ref{eqn:powersumequation}).

When $D$ is acyclic, $\U_D$ has the following compact form. This is the so-called \emph{symmetric function determinant} of $\overline{A}$ as coined by Stanley in \cite[Exercise 123]{EC2}.
\begin{corollary}\label{cor:uppertriangular}
Let $D$ be an acyclic digraph on $[n]$. Then
\[
\U_D = \sum_{\sigma \in \mathfrak{S}_n} p_{\mathrm{cyc}(\sigma)} \prod_{i=1}^n \overline{A}_{i,\sigma(i)}.
\]
In particular, $\U_D$ is $p$-positive.
\end{corollary}
\begin{proof}
If $D$ is acyclic then in Theorem~\ref{thm:main}(c) we have $p_{\mathrm{cyc}(\tau)} \sgn(\tau) \prod_{i \in I} A_{i,\tau(i)}=0$ for $I \neq \emptyset$. This leaves the desired sum.
\end{proof}

\begin{example}\label{ex:star}
Let $\bigcup_{i=1}^k V_i$ be a partition of $[n]$ and consider the digraph $D$ on $[n]$ with edge set $\bigcup_{i > j} V_i \times V_j$. Then $D$ is acyclic, so by Corollary~\ref{cor:uppertriangular} we have
\[
\U_D = \sum_{\sigma \in \mathfrak{S}_n} p_{\cyc(\sigma)} \prod_{i=1}^n \comp{A}_{i,\sigma(i)} = \sum_{(\sigma_1,\ldots,\sigma_k) \in \mathfrak{S}_{V_1} \times \cdots \times \mathfrak{S}_{V_k}} p_{\cyc(\sigma_1)} \cdots p_{\cyc(\sigma_k)}
\]
because for any $\sigma \in \mathfrak{S}_n$ the product $\prod_{i=1}^n \comp{A}_{i,\sigma(i)}$ is nonzero if and only if each cycle in $\sigma$ is individually contained completely in one of the sets $V_i$.  

Observe in this case we can write $\U_D$ very simply in terms of complete symmetric functions. Using the well-known identity $m! h_m = \sum_{\sigma \in \Sym_m} p_{\cyc(\sigma)}$, we find that
\[\U_D =\prod_{i=1}^k \sum_{\sigma \in \mathfrak{S}_{V_i}} p_{\cyc(\sigma)}
= (\lambda_1! \lambda_2! \cdots) h_{\lambda},
\]
where $\lambda$ is the partition of $n$ obtained from rearranging $|V_1|,\ldots,|V_k|$ in nonincreasing order. 
\end{example}

If the vertices of an acyclic digraph are labeled in topological order (meaning they are listed so that if $(u,v)$ is a directed edge then $u$ is listed before $v$) then we recover, from Corollary~\ref{cor:uppertriangular}, 
 the following expansion of $\U_D$ that was conjectured by Stanley and proved by Gessel; see \cite[Exercise 120c]{EC2}.  After some unpacking, our approach is effectively equivalent to Gessel's. 

Recall that a \emph{record}  of a permutation $\sigma \in \Sym_n$ is an index $r \in [n]$ such that $\sigma(r) > \sigma(i)$ for all $i < r$.  
If $r_1 < \cdots < r_j$ are the records of $\sigma$, then the \emph{record partition} of $\sigma$ is the partition of $n$ whose parts are the differences $\{r_2-r_1, r_3-r_2,\ldots, (n+1)-r_j\}$.  For example, $\sigma=325641 \in \Sym_6$ has records $\{1,3,4\}$ and record partition $(3,2,1)$.

\begin{corollary}
\label{cor:acyclicppositive}
Suppose $D$ is a digraph on $[n]$ such that $i > j$ for every directed edge  $(i,j)$. Then
\[
\U_D = \sum_{\substack{\sigma \in \mathfrak{S}_n \\ \Des_D(\sigma)=\emptyset}} p_{\mathrm{rp}(\sigma)}
\]
where $\mathrm{rp}(\sigma)$ is the record partition of $\sigma$.
\end{corollary}
\begin{proof}
Evidently $D$ is acyclic, so Corollary~\ref{cor:uppertriangular} gives $\U_D = \sum_{\sigma \in \mathfrak{S}_n} p_{\cyc(\sigma)} \prod_{i=1}^n \overline{A}_{i,\sigma(i)}$. The only nonvanishing summands are $p_{\cyc(\sigma)}$ for permutations $\sigma$ whose cycles are all $\comp{D}$-cycles. 
Now transform any such $\sigma$ into the following linear form $\sigma'$: Write each cycle of $\sigma$  with largest element first, and then concatenate the cycles in ascending order of greatest element. For example, $\sigma=(1\,4\,6)(2\,3)(5)$ becomes $\sigma'=325641$. (The mapping $\sigma \to \sigma'$ is a restriction of  Foata's fundamental bijection.)  Then  $\mathrm{cyc}(\sigma)=\mathrm{rp}(\sigma')$, and each consecutive pair $(\sigma'_j,\sigma'_{j+1})$ is either of the form $(i,\sigma(i))$ for some $i$ or satisfies $\sigma'_{j} < \sigma'_{j+1}$.  In particular, each $(\sigma'_j,\sigma'_{j+1})$ is a directed edge of $\comp{D}$, so $\Des_D(\sigma')=\emptyset$.\end{proof}


Finally, we have the following result of Grinberg and Stanley which ensures $p$-positivity of $\U_D$ when $D$ has no $2$-cycles.

\begin{corollary}\cite[Theorem 1.41]{GrinbergStanley}
Suppose $D$ is a digraph on $[n]$ and has no $2$-cycles. Then $\U_D$ is $p$-positive.    
\end{corollary}
\begin{proof}
Recalling that, working in the ring where $X^2=0$ (which we can do by Theorem~\ref{thm:main}), the quantities $\tr(X\comp{A})^k$ and $\tr(XA)^k$ enumerate the $k$-cycles in $\comp{D}$ and $D$, respectively (both with multiplicity $k$).  Suppose $D$ has no 2-cycles.  If arc $(u,v)$ appears in $D$ then the reverse arc $(v,u)$ appears in $\comp{D}$. Thus  reversing any $k$-cycle in $D$ yields a $k$-cycle in $\comp{D}$. In particular, $\tr(X\comp{A})^k$ contains all the terms of $\tr(XA)^k$, so all coefficients of $\tr(X\comp{A})^k +(-1)^{k-1} \tr(XA)^k$ are nonnegative.  Hence $\U_D$ is $p$-positive by Theorem~\ref{thm:main}(b).
\end{proof} 

\subsection{Tournaments} 
Grinberg and Stanley \cite{GrinbergStanley}  expand $\U_D$ when $D$ is a tournament, showing not only that $\U_D$ is $p$-positive in this case, but that $\U_D \in \mathbb{Z}[p_1,2p_3,2p_5,2p_7,\dots]$. We can recover this theorem directly from  Theorem~\ref{thm:main}.

\begin{corollary}\label{cor:tournament}\cite[Theorem 1.39]{GrinbergStanley}
Let $D$ be a tournament on $[n]$ and for $\sigma \in \mathfrak{S}_n$ let $\psi(\sigma)$ be the number of nontrivial cycles of $\sigma$. Then
\[
\U_D = \sum_{\substack{\sigma \in \mathfrak{S}(D) \\ \text{all cycles of } \sigma \text{ have odd length }}} 2^{\psi(\sigma)} p_{\cyc(\sigma)}
\]
\end{corollary}
\begin{proof}
Recall that Theorem~\ref{thm:main}(b) gives
\[
\U_D = \mathfrak{L}_n\, \exp \Big( \sum_k \frac{1}{k}p_k\big(\tr(X\comp{A})^k + (-1)^{k-1}  \tr(XA)^k\big)\Big).
\]
We again proceed in the quotient ring $\Lambda_{\vec{z}}[[x_1,\ldots,x_n]]/\langle x_1^2, \ldots, x_n^2 \rangle$, \goaway{Here we have $X^2=0$ and thus
$$
    (AX+X)^k= (AX)^k + X(AX)^{k-1}.
$$}
in which we claim
$$
    \tr (X\comp{A})^k =
    \begin{cases}
        \tr X  &\text{if $k=1$}  \\
        \tr (XA)^k &\text{otherwise}.
    \end{cases}
$$
Indeed,  $\comp{A}$ has all ones on the diagonal because $D$ has no loops, so $\tr(X\comp{A})=\tr X$. Since $D$ is a tournament, $D^{op}$ is simply $\comp{D}$ with its loops removed. Thus any $D$-cycle is the reversal of a non-loop $\comp{D}$-cycle, and vice-versa.  When $k > 1$, loops contribute to neither $\tr(X\comp{A})^k$ nor $\tr(XA)^k$ (since $X^2=0$), and so in this case we have $\tr(X\comp{A})^k=\tr(XA)^k$.

\goaway{Indeed for $k=1$ we have $\tr AX = 0$ since tournaments have no loops, whereas when $k>1$ we can apply the general identity $\tr(PQ)=\tr(QP)$ to get
\[ \tr(AX+X)^k=\tr(XA)^k+\underbrace{\tr((AX)^{k-1}X)}_{0}.
 \]}
We are left with
\begin{align*}
\U_D 
 &= \mathfrak{L}_n \exp\Big(p_1 \tr X \Big) \exp \left( \sum_{\text{odd $k > 1$}} \frac{1}{k}(2p_k) \tr(XA)^k\right) \\
 &= \mathfrak{L}_n \prod_i (1+p_1x_i) \cdot  \det H_{\vec{z}}(XA) \Big|_{p_1 \rightarrow 0, \, p_{2i} \rightarrow 0,  \, p_{2i+1}\rightarrow 2p_{2i+1}} \\
 &= \sum_{I \subseteq [n]} p_1^{n-|I|}\cdot \mathfrak{L}_I  \det H_{\vec{z}}(XA) \Big|_{p_1 \rightarrow 0, \, p_{2i} \rightarrow 0,  \, p_{2i+1}\rightarrow 2p_{2i+1}}  
\end{align*}
But by Lemma~\ref{lem:coefficientextraction}, for $I \subseteq [n]$ we have
$$
    \mathfrak{L}_I\, \det H_{\vec{z}}(XA) = \sum_{\sigma \in \mathfrak{S}_I} p_{\mathrm{cyc}(\sigma)} \prod_{i \in I} A_{i,\sigma(i)}
    = \sum_{\substack{\sigma \in \mathfrak{S}_I \\ \text{ every cycle in $\sigma$  is a  $D$-cycle}}}
    p_{\mathrm{cyc}(\sigma)}
$$    
and the result follows.
\end{proof}

It was observed in \cite{GrinbergStanley} that  Corollary~\ref{cor:tournament}  specifies to the following  expression for the number of Hamiltonian paths in the complement of a tournament:

\begin{corollary}\label{cor:hampaths}\cite[Theorem 1.39 \& Lemma 6.5]{GrinbergStanley} For a tournament $D$ on $[n]$,
\[
\mathrm{ham}(\overline{D})=\sum_{\substack{\sigma \in \mathfrak{S}(D) \\ \text{all cycles of } \sigma \text{ have odd length } }} 2^{\psi(\sigma)}.
\]
\end{corollary}

\begin{proof}
From~\eqref{eqn:UD_WD} we see that setting $z_1=1$ and $z_i=0$ for $i > 1$ transforms $\U_D$ into $\mathfrak{L}_n W_{\comp{D}}(1)$, which is precisely $\mathrm{ham}(\comp{D})$.  Evidently these transformations are equivalent to mapping $p_i \rightarrow 1$ for all $i$, and the result follows from Corollary~\ref{cor:tournament}.
\end{proof}

\subsection{Schur Function Expansions}\label{sec:schur}

We now explore the expansion of $\U_D$ in the Schur function basis. A motivating question is determining conditions on $D$ that ensure $\U_D$ is Schur positive. For instance, we see that the digraphs in Example~\ref{ex:star} are all Schur positive. Indeed for such a digraph $D$, if the partition of the vertex set is $V_1,V_2,\ldots,V_k$ with $|V_i|=\lambda_i$, and we set $\lambda$ to be the partition of $n$ with the $\lambda_i$'s listed in weakly decreasing order, then $\U_D = \left(\prod_{i=1}^k \lambda_i! \right) \sum_{\mu} K_{\mu,\lambda} s_{\mu} $ where $K_{\mu,\lambda}$ are the Kostka numbers which are known to be nonnegative.
However, $\U_D$ may fail to be Schur positive even if $D$ is acyclic. For example, consider the digraph $D$ on vertices $\{1,2,3,4\}$ with directed edges $\{(4,3), (3,2), (3,1)\}$.  Then $D$ is acyclic (in fact, a directed tree), and one can calculate that
\begin{equation*}
    \U_D = 10s_{(4)}+4s_{(3,1)} - 2s_{(2,2)}+2s_{(2,1,1)}.
\end{equation*}

We begin our  investigation by developing two general expressions concerning the resolution of $\U_D$ in the Schur basis.  The first of these identifies the coefficient of $s_{\lambda}$ in $\U_D$ as a Jacobi-Trudi determinant. 

For any digraph $D$ on $[n]$,  and for any $k \geq 0$, let 
\[
\xi_{k+1}(D):=\sum_{\substack{(i_0,i_1, \ldots, i_{k})\\ \text{a path in }D}} x_{i_0}x_{i_1} \cdots x_{i_{k}} \in  \mathbb{Q}[x_1,\ldots,x_n]
\]
be the generating polynomial for paths of length $k$ in $D$. Further set $\xi_0(D)=1$ and  take $\xi_k(D)=0$ for $k < 0$.  Note that $\xi_{k+1}(D)$ is obtained from $\gamma_{k+1}(D)$ by setting $x_i^2=0$ for all $i$. Finally, for $\lambda =(\lambda_1,\lambda_2,\ldots) \vdash n$,  define  $P_{\lambda}(D)$ to be the $n \times n$ Jacobi-Trudi matrix whose $(i,j)$-th entry is $\xi_{\lambda_i-i+j}({D})$.  

\begin{proposition} \label{Schurexpansion}
Let $D$ be any digraph on $[n]$. Then
$$
[s_{\lambda}]\, \U_D = 
\mathfrak{L}_n \det P_{\lambda}(\comp{D}) = \mathfrak{L}_n \det P_{\lambda^T}(D).
$$
\end{proposition}

\begin{proof}
We have
\begin{align*}
\U_D = \mathfrak{L}_n \prod_{i \geq 1} W_{\comp{D}}(z_i) 
= \mathfrak{L}_n \prod_{i \geq 1} \Big( \sum_{k \geq 0} \xi_k(\comp{D}) z_i^k \Big) 
= \mathfrak{L}_n \sum_{\lambda \vdash n} h_\lambda(\vec{y}) m_{\lambda}
(\vec{z})
\Bigg|_{h_i(\vec{y}) \rightarrow \xi_i(\comp{D})},
\end{align*}
with the the last equality holding since $\prod_{i \geq 1} \sum_{k \geq 0} h_k(\vec{y})z_i^k=\sum_{\lambda \vdash n} h_{\lambda}(\vec{y})m_{\lambda}(\vec{z})$. The Cauchy identity~\eqref{eqn:cauchy} then gives
\begin{align*}
\U_D = 
\mathfrak{L}_n
\sum_{\lambda \vdash n} {s}_{\lambda}(\vec{y})s_{\lambda}(\vec{z}) \Big|_{h_i(\vec{y}) \rightarrow \xi_i({\comp{D}})},
\end{align*}
and the first expression for $[s_\lambda] \ \U_D$ follows by the Jacobi-Trudi formula~\eqref{eqn:jacobitrudi}.  Since $W_{\comp{D}}(z)=(W_D(-z))^{-1}$, 
 the identification $h_i(\vec{y}) \rightarrow \xi_i(\comp{D})$ is equivalent to
$e_i(\vec{y}) \rightarrow \xi_i(D)$, so the second expression  follows similarly.
\end{proof}

For example, returning to the digraph $D$ in Figure~\ref{fig:pathcycle} we have
\begin{align*}
    \xi_1(D) = x_1+x_2+x_3,
    \qquad
    \xi_2(D) = x_1 x_3 + x_3 x_2, 
    \qquad
    \xi_3(D) = x_1 x_3 x_2,
\end{align*}
with $\xi_0(D)=1$ and $\xi_k(D)=0$ for $k > 3$.  Proposition~\ref{Schurexpansion} then gives
\begin{align*}
\U_D
&=
\mathfrak{L}_3
 \overbrace{
 \left|\begin{matrix} \xi_3 & 0 & 0 \\ 0 & 1 & \xi_1 \\ 0 & 0 & 1\end{matrix} \right|
 }^{\det P_{(3)}}\cdot s_{(1^3)}
+
\mathfrak{L}_3
\overbrace{
\left| \begin{matrix} \xi_2 & \xi_3 & 0 \\ 1 & \xi_1 & \xi_2 \\ 0 & 0 & 1\end{matrix} \right| 
}^{\det P_{(2,1)}}\cdot s_{(2,1)}
+
\mathfrak{L}_3
\overbrace{
\left| \begin{matrix} \xi_1 & \xi_2 & \xi_3 \\ 1 & \xi_1 & \xi_2 \\ 0 & 1 & \xi_1\end{matrix} \right|
}^{\det P_{(1^3)}} \cdot s_{(3)} \\
&=
s_{(1^3)} + s_{(2,1)} + 3s_{(3)},
\end{align*}
in agreement with~\eqref{eqn:monomialequation} and~\eqref{eqn:powersumequation}.

We can also see using Theorem~\ref{thm:main} that $\U_D$ can be written in the Schur basis with coefficients expressed in terms of immanants of submatrices of $A$ and $\comp{A}$ and Littlewood-Richardson coefficients
$c^{\nu}_{\lambda \mu}$  (see \cite{EC2} for details).

\begin{proposition}\label{prop:UDexpansion}
For any digraph $D$ on $[n]$ we have
   \[\U_D=\sum_{\nu \vdash n} \left(  \sum_{\substack{I \subseteq [n] \\ \lambda \vdash |I^c|, \ \mu \vdash |I|}}  \  \mathrm{Imm}_{\lambda}\overline{A}[I^c] \cdot \mathrm{Imm}_{\mu^T}A[I] \cdot   c^{\nu}_{\lambda \mu}  \right) s_{\nu}\] 
   where $\{c^{\nu}_{\lambda \mu}\}$ are Littlewood-Richardson coefficients defined by
   $s_{\lambda} s_{\mu} = \sum_{\nu \vdash |\lambda|+|\mu|} c^{\nu}_{\lambda \mu} s_{\nu}$.
\end{proposition}
\begin{proof}
If $I \subseteq [n]$, $\mu$ is a partition of $|I|$ and $\tau \in \Sym_I$ then $\chi^{\mu}(\tau) \sgn(\tau)=\chi^{\mu^T}(\tau)$. Using Theorem~\ref{thm:main}(c) and Equation~\eqref{eqn:powertoschur} we therefore get
\begin{align*}
\U_D &= \sum_{I \subseteq [n]} \left( \sum_{\sigma \in \Sym_{I^c}}
p_{\cyc(\sigma)} \prod_{i \in I^c} \overline{A}_{i, \sigma(i)} \right)
\left( \sum_{\tau \in \Sym_{I}}
p_{\cyc(\tau)} \sgn(\tau) \prod_{i \in I} A_{i, \tau(i)} \right)  \\
&=\sum_{I \subseteq [n]} \left( \sum_{\lambda \vdash |I^c|} \left( \sum_{\sigma \in \Sym_{I^c}} \chi^{\lambda}(\sigma)  \prod_{i \in I^c} \overline{A}_{i, \sigma(i)} \right)  s_{\lambda} \right) \cdot \left( \sum_{\mu \vdash |I| } \left( \sum_{\tau \in \Sym_{I}} \underbrace{\chi^{\mu}(\tau) \sgn(\tau)}_{\chi^{\mu^T}(\tau)}   \prod_{i \in I} A_{i, \tau(i)} \right)  s_{\mu} \right) \\
&=\sum_{I \subseteq [n]} \left( \sum_{\lambda \vdash |I^c|} \text{Imm}_{\lambda}\overline{A}[I^c]  \cdot s_{\lambda} \right) \left( \sum_{\mu \vdash |I|} \text{Imm}_{\mu^T}A[I] \cdot s_{\mu} \right).
\end{align*}
The result follows.
\end{proof}

The expression for $\U_D$ in Proposition~\ref{prop:UDexpansion} simplifies when $D$ is acyclic.

\begin{proposition}\label{prop:uppertriangularschur}
Let $D$ be an acyclic digraph on $[n]$. Then
\[
\U_D = \sum_{\lambda \vdash n} \mathrm{Imm}_{\lambda}(\overline{A}) \cdot s_{\lambda}.
\]
\end{proposition}
\begin{proof}
From Corollary~\ref{cor:uppertriangular} and  Equation~(\ref{eqn:powertoschur}), we have
\[
\U_D = \sum_{\sigma \in \mathfrak{S}_n} p_{\cyc(\sigma)} \prod_{i=1}^n \overline{A}_{i,\sigma(i)} = \sum_{\sigma \in \mathfrak{S}_n} \sum_{\lambda \vdash n} \chi^{\lambda}(\sigma)  s_{\lambda} \prod_{i=1}^n \overline{A}_{i,\sigma(i)} =  \sum_{\lambda \vdash n} \left(\sum_{\sigma \in \mathfrak{S}_n} \chi^{\lambda}(\sigma) \prod_{i=1}^n \overline{A}_{i,\sigma(i)} \right) s_{\lambda}
\]
and the equation follows.
\end{proof}
Proposition~\ref{prop:uppertriangularschur} shows that for an acyclic digraph $D$ on $[n]$, $\U_D$ is Schur positive if and only if  $\mathrm{Imm}_{\lambda}(\comp{A})$ is nonnegative for every $\lambda \vdash n$. By a well-known result of Stembridge~\cite{StembridgeImmanants}, all of these immanants are known to be nonnegative in the special case when $\comp{A}$ is totally nonnegative (i.e. all minors of $\comp{A}$ are nonnegative). However,  using \cite[Theorem 2.1]{brualdi} one can see that the tandem conditions of $D$ being acyclic and $\comp{A}$ being totally nonnegative are very restrictive. For instance, these conditions are satisfied in Example~\ref{ex:star}.  The next example describes more general instances when this happens, and relates to a celebrated theorem on chromatic symmetric functions.

\begin{example} Let $P$ be a partial order on $[n]$ and $Y_P=\{(i,j) : i >_P j\}$. Then it is known that $\U_{Y_P} = X_{\text{inc}(P)}$, the chromatic symmetric function of the incomparability graph of $P$. One can use this to piece together the celebrated theorem of Gasharov \cite{SchurPositivity} that $X_{\text{inc}(P)}$ is Schur positive if $P$ is $(\bf{3}+\bf{1})$-free. Observe that $Y_P$ is acyclic, so by Proposition~\ref{prop:uppertriangularschur}, $X_{\text{inc}(P)} = \sum_{\lambda \vdash n} \text{Imm}_{\lambda}(\overline{A(Y_P)}) s_{\lambda}.$  Guay-Paquet \cite{guay2013modular} establishes that to prove Schur positivity for $(\bf{3}+\bf{1})$-free posets it is sufficient to prove it for \emph{unit interval orders}, that is posets that are $(\bf{3}+\bf{1})$- and $(\bf{2}+\bf{2})$-free. Now by  Skandera and Reed \cite[Proposition 5]{skandera2003total} together with Brualdi \cite[Theorem 2.2]{brualdi}, $\comp{A(Y_P)}^T$ is totally nonnegative and hence $A(Y_P)$ is. We deduce again by Stembridge \cite{StembridgeImmanants} that $X_{\mathrm{inc}(P)}=\U_{Y_P}$ is Schur positive.
\end{example}

\begin{remark} Comparing Propositions~\ref{Schurexpansion} and~\ref{prop:uppertriangularschur} we see that for an acyclic digraph $D$ we have
$$\mathrm{Imm}_{\lambda}(\comp{A}) = \mathfrak{L}_n \det 
[\overline{\gamma}_{\lambda_i-i+j}]
$$ 
where $\overline{\gamma}_k:=\gamma_k(\comp{D})$ is given by
$$
    \overline{\gamma}_k 
    = [z^k] W_{\comp{D}}(z)
    = [z^k] \det(I-zX\comp{A})^{-1}.
$$
Similar computations show that the immanants of an arbitrary matrix $M$ can be expressed in the Jacobi-Trudi form  $\mathrm{Imm}_{\lambda}(M) =\mathfrak{L}_n \det[\Delta_{\lambda_i-i+j}]$, where $\Delta_k = [z^k] \det(I-zXM)^{-1}$.  See~\cite[Theorem 2.1]{goulden1992immanants}.
\end{remark}

We now investigate combinatorial features of the coefficients in the expansion of $\U_D$ in the Schur basis. We start by observing that there is a direct combinatorial interpretation for the coefficient of $s_{\lambda}$ when $\lambda$ is a \emph{hook shape}, i.e. $\lambda=(i,1^{n-i})$ for some $1 \leq i \leq n$.

\begin{proposition}\label{prop:hookcoefficients}
Let $D$ be a digraph on $[n]$, $1 \leq i \leq n$, and $\lambda$ be the hook partition $\lambda=(i,1^{n-i})$. Then $[s_{\lambda}] \ \U_D$ is the number of permutations $\sigma \in \mathfrak{S}_n$ such that $\Des_D(\sigma)=\{i,i+1,\ldots,n-1\}$. In particular, $[s_{(1^n)}] \ \U_D$ is the number of Hamiltonian paths in $D$ and $[s_{(n)}] \ \U_D$ is the number of Hamiltonian paths in $\comp{D}$.
\end{proposition}
\begin{proof}
For a composition $\alpha=(\alpha_1,\ldots,\alpha_{k}) \models n$, define $S(\alpha) = \{\alpha_1,\alpha_1+\alpha_2,\ldots,\alpha_1 + \cdots +\alpha_{k-1}\} \subseteq [n-1]$. Then
\[
\U_D = \sum_{\alpha \models n} c_\alpha F_{S(\alpha)}
\]
where $c_{\alpha}$ is the number of permutations $\sigma \in \Sym_n$ with $\Des_D(\sigma) = S(\alpha)$.
Since we know $\U_D$ is symmetric, \cite[Theorem 15]{egge2010quasisymmetric} implies that $[s_{\lambda}]\,\U_D=c_{\lambda}$ for any hook partition  $\lambda \vdash n$. The result follows since $\lambda=(i,1^{n-i})$ has $S(\lambda)=\{i,i+1,\ldots,n-1\}$.
\end{proof}

The following Schur expansion of $\U_D$  when $D=\{(2,1),(3,2),\ldots,(n,n-1)\}$ is given in \cite[Exercise 120]{EC2} and was also observed by Chow for general directed Hamiltonian paths $D$ (see \cite[Section 6]{ChowPathCycle}):
\[
\U_D = \sum_{i=1}^n f_i \cdot s_{(i,1^{n-i})}.
\]
Here, $f_i$ is the number of permutations in $\mathfrak{S}_i$ with no reverse successions: that is permutations $\pi \in \mathfrak{S}_i$ with no index $j$ such that $\pi_{j+1}=\pi_j-1$. Note that the coefficients  agree with Proposition~\ref{prop:hookcoefficients}. Indeed, according to that proposition,  $[s_{(i,1^{n-i})}]\,\U_D$ is the number of permutations in $\Sym_n$ with reverse successions precisely at positions $\{i,i+1,\ldots,n-1\}$. Now there is a bijection between such permutations $\sigma$ and permutations in $\Sym_i$ with no reverse succession by sending $\sigma$ to the standardization of $\sigma_1 \sigma_2 \cdots \sigma_i$ (and one can readily see this is reversible).

\begin{remark}
Our proof of Proposition~\ref{prop:hookcoefficients} relied on~\cite[Theorem 15]{egge2010quasisymmetric}, which states (in part) that $\left[  s_{\lambda}\right]  f=\left[  F_{S\left(  \lambda\right)  }\right]f$ for any symmetric function $f \in \Lambda_\vec{z}$ and any hook partition $\lambda=(i,1^{n-i})$.   Grinberg~\cite{grinbergprivate} has pointed out the following alternative explanation for this identity.

 The Hall inner product on
$\Lambda$ $(=\Lambda_{\vec{z}})$ is a restriction of a bilinear form $\left\langle \cdot
,\cdot\right\rangle $ between the Hopf algebras $\operatorname*{NSym}$ and $\operatorname*{QSym}$ of noncommutative and quasisymmetric functions. The ribbon basis $\{R_{\alpha}\}$ of $\mathrm{NSym}$ and fundamental basis $\{F_{\alpha}\}$ of $\mathrm{QSym}$ are dual with respect to this pairing, and the canonical injection $\iota:\Lambda\rightarrow\operatorname*{QSym}$ and projection $\pi:\operatorname*{NSym}\rightarrow \Lambda$ are mutually adjoint.  
Since $s_{\lambda}=\pi(R_{\lambda})$ for ribbon-shaped $\lambda$, we have 
\begin{align*}
\left[  s_{\lambda}\right]  f   =\left\langle s_{\lambda},f\right\rangle
=\left\langle \pi\left(  R_{\lambda}\right)  ,f\right\rangle 
=\left\langle R_{\lambda},\iota\left(  f\right)  \right\rangle
 =\left[  F_{S\left(  \lambda\right)  }\right]  f.
 \end{align*}
See~\cite[Section 5.4]{grinberg2014hopf} for details concerning the above assertions.
\end{remark}

\subsection{Extensions to the Chow Path-Cycle Symmetric Function} Several of our results regarding $\U_D$ can be extended to Chow's symmetric function $\Xi_{D}(\vec{z},\vec{y})$. We start by noting the action of the fundamental involution, which we now denote $\omega_{\vec{z}}$ to emphasize that it is acting only on the $\vec{z}$ variables.  The following result appeared as \cite[Theorem 1]{ChowPathCycle} and was  reproved by Lass~\cite{Lass} using an approach equivalent to ours.

\begin{theorem}\label{thm:FullChow}\cite[Theorem 1]{ChowPathCycle} For a digraph $D$ on $[n]$ we have
\[[\omega_{\vec{z}} \ \Xi_{\comp{D}}(\vec{z},-\vec{y})]_{\vec{z} \to (\vec{z},\vec{y})} = \Xi_D(\vec{z},\vec{y}),\]
where the operation $\vec{z} \to (\vec{z},\vec{y})$ replaces $\vec{z}$ with the union of variables $\vec{z}$ and $\vec{y}$.
\end{theorem}

\begin{proof}
Without loss of generality, assume $D$ has vertex set $[n]$. Recall that \[\Xi_D(\vec{z},\vec{y})=\sum_S \tilde{m}_{\text{path}(S)}(\vec{z})p_{\text{cycle}(S)}(\vec{y})\] where the sum is over all path-cycle covers of $D$. Notice that for a given summand $S$, the vertices appearing in the parts of $\text{path}(S)$ are  complementary to those appearing in the parts of $\text{cycle}(S)$. From this we see that
\begin{align*}
\Xi_D(\vec{z},\vec{y})&=\sum_S \tilde{m}_{\text{path}(S)}(\vec{z})p_{\text{cycle}(S)}(\vec{y})\\
&=\sum_{I \subseteq [n]} \left( \sum_{S' \text{ a path cover of } D[I]} \tilde{m}_{\text{path}(S')}(\vec{z}) \right) \left( \sum_{S'' \text{ a cycle cover of } D[I^c]} p_{\text{cycle}(S'')}(\vec{y}) \right) 
\end{align*}
where $D[I]$ and $D[I^c]$ denote the subgraphs of $D$ induced by $I$ and $I^c$, respectively.

From Corollary~\ref{cor:cyclecover}, for any $I \subseteq [n]$ we have
\[
\sum_{S'' \text{ a cycle cover of } D[I^c]} p_{\text{cycle}(S'')}(\vec{y}) =  \mathfrak{L}_{I^c} \det H_{\vec{y}}(XA).
\]
On the other hand, we have
\[
\sum_{S' \text{ a path cover of } D[I]} \tilde{m}_{\text{path}(S')}(\vec{z}) = \mathfrak{L}_I W_{D[I]}(z_1)W_{D[I]}(z_2) \cdots
\]
As noted in Section~\ref{sec:walks}, $W_{D[I]}(z)$ is obtained from $W_D(z)$  by setting $x_j=0$ for $j \not \in I$.  Thus we obtain 
\begin{align}
\sum_{S' \text{ a path cover of } D[I]} \tilde{m}_{\text{path}(S')}(\vec{z}) &=
 \mathfrak{L}_I \Bigg(W_D(z_1)W_{D}(z_2) \cdots \Big|_{x_j=0,\, j \not \in I}\Bigg) & \notag\\
 &= \mathfrak{L}_I W_D(z_1)W_{D}(z_2) \cdots  &\notag\\
 &= \mathfrak{L}_I \det H_{\bf{z}}(XA) \det E_{\bf{z}}(X\comp{A}), \label{eq:pathcover}
\end{align}
where the last line follows from Lemma 1 just as in the proof of Theorem 10(a). 
 Altogether this gives
\begin{align*}
\Xi_D(\vec{z},\vec{y})
&=\sum_{I \subseteq [n]} \mathfrak{L}_I \det H_{\vec{z}}(XA) \det E_{\vec{z}}(X\comp{A}) \cdot \mathfrak{L}_{I^c} \det H_{\vec{y}}(XA)  \\
&= \mathfrak{L}_n \det H_{\vec{z}}(XA) E_{\vec{z}}(X\comp{A})H_{\vec{y}}(XA) \\
&= \mathfrak{L}_n \det H_{(\vec{z},\vec{y})}(XA) E_{\vec{z}}(X\comp{A}),
\end{align*}
from which we obtain
\begin{align*}
[\omega_{\vec{z}} \ \Xi_{\comp{D}}(\vec{z},-\vec{y})]_{\vec{z} \to (\vec{z},\vec{y})}
&= [\omega_{\vec{z}}\ \mathfrak{L}_n \det \left(H_{\vec{z}}(X\comp{A}) H_{-\vec{y}}(X\comp{A}) E_{\vec{z}}(XA)\right)]_{\vec{z} \to (\vec{z},\vec{y})} \\
&= \mathfrak{L}_n \det E_{(\vec{z},\vec{y})}(X\comp{A}) H_{-\vec{y}}(X\comp{A}) H_{(\vec{z},\vec{y})}(XA) \\
&=\mathfrak{L}_n \det H_{(\vec{z},\vec{y})}(XA) E_{\vec{z}}(X\comp{A}) = \Xi_D(\vec{z},\vec{y}).
\end{align*}
\end{proof}
In \cite[Section 6]{ChowPathCycle}, Chow also defines
\[
\hat{\Xi}_D(\vec{z},\vec{y}) = \sum_S (-2)^{\ell(\cycle{S})} \tilde{m}_{\path{S}}(\vec{z},\vec{y})\, p_{\cycle{S}}(\vec{y})
\]
where $\ell(\cycle{S})$ is the number of cycles in $S$ and the sum is over all path-cycle covers $S$ of $D$. 
The series $\hat{\Xi}$  was introduced because of the following involutive property:
\begin{theorem}\cite[Section 6]{ChowPathCycle} For a digraph $D$ on $[n]$ we have,
\[\omega_{\vec{z}} \ \hat{\Xi}_D(\vec{z},-\vec{y})=\hat{\Xi}_{\comp{D}}(\vec{z},\vec{y}).\]
\end{theorem}
\begin{proof}
Let $I \subseteq [n]$.
Substituting $\vec{z} \rightarrow (-\vec{y},-\vec{y})$ in
Corollary~\ref{cor:cyclecover} 
yields
\begin{align*}
\sum_{S \text{ a cycle cover of } D[I]} (-2)^{\ell(\cycle{S})} p_{\cycle{S}}(\vec{y}) 
=  \mathfrak{L}_I \det H_{\vec{y}}(XA)^{-2},
\end{align*}
whereas substituting $\vec{z}\rightarrow(\vec{z},\vec{y})$ in equation~\eqref{eq:pathcover} gives
\begin{equation}
\sum_{S' \text{ a path cover of } D[I]} \tilde{m}_{\text{path}(S')}(\vec{z},\vec{y}) 
= \mathfrak{L}_I \det H_{\bf{z}}(XA)H_{\bf{y}}(XA) \det E_{\bf{z}}(X\comp{A})
E_{\bf{y}}(X\comp{A}).
\end{equation}
There follows
\[
\hat{\Xi}_D(\vec{z},\vec{y})= \mathfrak{L}_n\ \dfrac{\det H_{\vec{z}}(XA) E_{\vec{z}}(X\comp{A})E_{\vec{y}}(X\comp{A})}{\det H_{\vec{y}}(XA)},
\]
and consequently
\begin{align*}
\omega_{\vec{z}} \ \hat{\Xi}_D(\vec{z},\vec{-y}) &=
 \mathfrak{L}_n\dfrac{\det E_{\vec{z}}(XA) H_{\vec{z}}(X\comp{A})E_{-\vec{y}}(X\comp{A})}{\det H_{-\vec{y}}(XA)} \\
&=
 \mathfrak{L}_n\dfrac{\det E_{\vec{z}}(XA) H_{\vec{z}}(X\comp{A})E_{\vec{y}}(XA)}{\det H_{\vec{y}}(X\comp{A})} \\
&=\hat{\Xi}_{\comp{D}}(\vec{z},\vec{y}),
\end{align*}
since $H_{-\vec{y}}(t)=1/E_{\vec{y}}(t)$ and  $E_{-\vec{y}}(t)=1/H_{\vec{y}}(t)$.
\end{proof}

Finally, we partially address a question of Grinberg and Stanley~\cite[Question 1.18]{GrinbergStanley} by generalizing Corollary~\ref{cor:U_Dpowersum} to expand the Chow path-cycle symmetric function in terms of power sums.

\begin{theorem}\label{thm:Chowpowersum} Let $D$ be a digraph on $[n]$. Then
\[
\Xi_D(\vec{z},\vec{y}) = \sum_{I \subseteq [n]} \sum_{\substack{\sigma \in \Sym_I(\comp{D}) \\ \tau \in \Sym_{I^c}(D)}} \sgn(\sigma) \cdot p_{\cyc(\sigma)}(\vec{z}) \cdot p_{\cyc(\tau)}(\vec{z},\vec{y}).
\]
\end{theorem}
\begin{proof}
\goaway{For a digraph $D$ on $[n]$, $I \subseteq [n]$, and $\sigma \in \mathfrak{S}_I$ define $\text{Cycs}(\sigma)$ to be the set of disjoint cycles in the disjoint cycle decomposition of $\sigma$, and let
\[
\phi(\sigma):=\sum_{\substack{\gamma \in \text{Cycs}(\sigma) \\ \gamma \text{ is a } \comp{D}-\text{cycle}}} (\ell(\gamma)-1)
\]
where $\ell(\gamma)$ is the length of $\gamma$. We first show 
\begin{equation}\label{eqn:product}
\Xi_D(\vec{z},\vec{y})=\sum_{\substack{I \subseteq [n] \\ \sigma \in \mathfrak{S}_I(D,\overline{D}) \\ \tau \in \mathfrak{S}_{I^c}(D)}} (-1)^{\phi(\sigma)} p_{\cyc(\sigma)}(\vec{z}) p_{\cyc(\tau)}(\vec{y}).
\end{equation}
}
As seen in the proof of Theorem~\ref{thm:FullChow},
\begin{align*}
\Xi_D(\vec{z},\vec{y})
&= \mathfrak{L}_n \det H_{\vec{z}}(XA) E_{\vec{z}}(X\comp{A})H_{\vec{y}}(XA)\\
&= \mathfrak{L}_n \det H_{(\vec{z},\vec{y})}(XA) E_{\vec{z}}(X\overline{A}) \\
&=\sum_{I \subseteq [n]} \mathfrak{L}_I \det E_{\vec{z}}(X\overline{A}) \cdot \mathfrak{L}_{I^c} \det H_{(\vec{z},\vec{y})}(XA) 
\end{align*}
Now applying Lemma~\ref{lem:coefficientextraction} we get 
\[
\Xi_D(\vec{z},\vec{y}) = \sum_{I \subseteq [n]} \sum_{\substack{\sigma \in \Sym_I \\ \tau \in \Sym_{I^c}}} \sgn(\sigma) \cdot p_{\cyc(\sigma)}(\vec{z}) \cdot p_{\cyc(\tau)}(\vec{z},\vec{y})  \prod_{i \in I^c} A_{i,\tau(i)} \prod_{j \in I} \comp{A}_{j,\sigma(j)}.
\]
The result follows since the product over $I^c$ and the product over $I$ are non-zero in a given summand if and only if $\sigma \in \Sym_I(\overline{D})$ and $\tau \in \Sym_{I^c}(D)$.
\end{proof}

\subsection*{Acknowledgements}

The authors thank the anonymous referees for fruitful recommendations that greatly improved the manuscript. The authors are very grateful to Darij Grinberg for meticulously reviewing an earlier draft of this manuscript and providing insightful comments, as well as for the constructive discussions that greatly contributed to its development. The authors also thank Alejandro H. Morales for fruitful discussions on the literature. The second author is partially supported by research funds from York University, and NSERC Discovery Grant \#RGPIN-2025-06304.


\end{document}